\documentclass{amsart}

\usepackage{amsmath, amssymb, tikz}
\usetikzlibrary{arrows}
\newtheorem{theorem}{Theorem}[section]
\newtheorem{lemma}[theorem]{Lemma}

\theoremstyle{definition}
\newtheorem{example}[theorem]{Example}

\newtheorem{definition}[theorem]{Definition}

\begin{document}

\title{Leavitt path algebras satisfying a polynomial identity}

\author{Jason Bell}
\address{Jason Bell\\
University of Waterloo\\
Department of Pure Mathematics\\
200 University Avenue West\\
Waterloo, Ontario \  N2L 3G1\\
Canada}
\email{jpbell@uwaterloo.ca}

\author{T. H. Lenagan}
\address{T. H. Lenagan, Maxwell Institute for Mathematical Sciences, 
School of Mathematics, University of Edinburgh, 
James Clerk Maxwell Building, King's Buildings, 
Mayfield Road, Edinburgh EH9 3JZ, Scotland, UK\\
}\email{tom@maths.ed.ac.uk}

\author{Kulumani M. Rangaswamy}
\address{Kulumani M. Rangaswamy\\
University of Colorado, Colorado Springs, Colorado 80919, USA\\
}
\email{krangasw@uccs.edu}

\thanks{The first author acknowledges support of NSERC grant 31-611456. 
The second author acknowledges support of EPSRC grant EP/K035827/1.}

\begin{abstract}
Leavitt path algebras $L$ of an arbitrary graph $E$ over a field $K$
satisfying a polynomial identity are completely characterized both in
graph-theoretic and algebraic terms. When $E$ is a finite graph, $L$
satisfying a polynomial identity is shown to be equivalent to the
Gelfand-Kirillov dimension of $L$ being at most one, though this is no longer
true for infinite graphs. It is shown that, for an arbitrary graph $E$, the
Leavitt path algebra $L$ has Gelfand-Kirillov dimension zero if and only if
$E$ has no cycles. Likewise, $L$ has Gelfand-Kirillov dimension one if and
only if $E$ contains at least one cycle, but no cycle in $E$ has an exit.

\end{abstract}

\maketitle


\section{Introduction}

Leavitt path algebras were introduced in \cite{AA, AMP} as algebraic
analogues of graph C*-algebras and as natural generalizations of Leavitt
algebras of type $(1,n)$ constructed by Leavitt \cite{L}. The various ring-theoretical
properties of these algebras have been actively investigated in a series of
papers, see, for example, \cite{AA, AAPS, ABR, AR, AMP}.\vskip 1ex

It is straightforward to show that a Leavitt path algebra $L$ of a connected
graph $E$ over a field $K$ (see Section~\ref{section-defs} for the relevant
definitions) is commutative if and only if the graph $E$ is either a single
vertex or consists of a single vertex $v$ and an edge $e$ which is a loop at
$v$; namely, the edge $e$ begins and ends at $v$. In this case, $L$ is
isomorphic to $K$ or to the Laurent polynomial ring $K[x,x^{-1}]$.
Observing that commutative algebras satisfy the polynomial identity $xy-yx=0$,
it is natural to ask under which conditions a Leavitt path algebra satisfies a
polynomial identity. In this paper, we obtain a complete characterization of
Leavitt path algebras satisfying a polynomial identity in both algebraic and
graph-theoretic terms (Theorem~\ref{General PI LPA}). In graph-theoretic
terms, we show that the Leavitt path algebra $L$ of an arbitrary graph $E$
over the field $K$ satisfies a polynomial identity if and only if no cycle in
$E$ has an exit and there is a fixed positive integer $d$ such that for every
vertex $v\in E$, the number of distinct paths ending at $v$ having no repeated
vertices is at most $d$. In this case, we show that $L$ is a subdirect product
of matrix rings of order $\leq d$ over $K$ and $K[x,x^{-1}]$. \vskip 1ex

If, in addition, $E$ is row-finite (that is,  each edge of 
$E$ emits only finitely many edges), we get a stronger conclusion
for $L$: the Leavitt path algebra $L$ is isomorphic to a (possibly infinite) direct sum of matrix rings 
either over $K$ or  $K[x,x^{-1}]$, where the order of each
matrix ring in this decomposition is less than a fixed positive integer $d$ 
and is specified in terms of graph-theoretic data 
(Theorem~\ref{Row finite PI LPA}). \vskip 1ex

Specialising further, 
when $E$
is a finite graph we obtain several equivalent characterizing properties for
$L$ to be a PI algebra (Theorem \ref{GK-Finite graph}) including the property
that the Gelfand-Kirillov dimension (GK dimension, for short) of $L$ is
at most $1$.  \vskip 1ex

In general, if $E$ is an infinite graph with the property that its associated
Leavitt path algebra $L$ is PI then $L$ must have GK dimension at most $1$. We
give examples, however, that show that there exist Leavitt path algebras
having GK dimension $\leq1$ which are not PI algebras. We then consider the
larger class of Leavitt path algebras having low GK dimension. For instance, a
Leavitt path algebra $L$ of a graph $E$ has GK dimension $0$ if and only if
$L$ is von Neumann regular, equivalently, $E$ has no cycles. Likewise, $L$
will have GK dimension $1$ if and only if $E$ contains at least one cycle and
no cycle in $E$ has an exit. In this case, $L$ is a directed union of finite
direct sums of matrix rings of finite order over $K$ and $K[x,x^{-1}]$.



\section{Background and definitions}\label{section-defs}
Here, we give some of the background needed for the paper along with an
overview of some of the earlier work on this subject. Unless otherwise stated,
all the graphs that we consider are arbitrary in the sense that no
restriction is placed either on the number of vertices or on the number of
edges emitted by any single vertex. Generally, we follow the notation and
terminology for Leavitt path algebras that appears in \cite{AA, AMP}. We
give below a short outline of some of the basic concepts and results that we
need.

\begin{definition} A (directed) \emph{graph} $E=(E^{0},E^{1},r,s)$ consists of two sets $E^{0}$ and
$E^{1}$ together with maps $r,s:E^{1}\rightarrow E^{0}$. The elements of
$E^{0}$ are called the \textit{vertices} of $E$ and the elements of $E^{1}$
are called the \textit{edges} of $E$. For each edge $e\in E^1$, there are
vertices $s(e),r(e)\in E^0$, not necessarily distinct, such that $e$ begins at
$s(e)$ and ends at $r(e)$. The element $r(e)$ is called the \emph{range} of
$e$ and the element $s(e)$ is called the \emph{source} of $e$.
\end{definition}

We now give some graph-theoretic terminology that will be useful. A vertex $v$
is called a \textit{sink} if it emits no edges (that is, $s^{-1}(v)$ is empty)
and is called an \textit{infinite emitter} if it emits infinitely many edges
(that is, $\#s^{-1}(v)=\infty$). 
A vertex that is neither a sink nor an infinite emitter is called a 
\textit{regular vertex}; that is, the regular vertices are the vertices 
which emit a nonzero finite number of edges. 
A\textit{ path} $\mu$ of length $n>0$ is a
finite sequence of edges $\mu=e_{1}e_{2} \cdots e_{n}$ with
$r(e_{i})=s(e_{i+1})$ for all $i=1,\ldots,n-1$. 
For such a path, we set $s(\mu):=s(e_1)$ and $r(\mu):=r(e_{n})$. 
We consider a vertex to be a path of length $0$. The set of all
vertices on the path $\mu$ is denoted by $\mu^{0}$. \vskip 1ex

A path $\mu=e_{1}\dots e_{n}$ in $E$ is \textit{closed} if $r(e_{n}%
)=s(e_{1})$, in which case $\mu$ is said to be based at the vertex
$v=s(e_{1})$. A closed path $\mu$ as above is called \textit{simple} provided
it does not pass through its base more than once; that is, $s(e_{i})\neq
s(e_{1})$ for all $i=2,\ldots,n$. The closed path $\mu$ is called a
\textit{cycle} if it does not pass through any of its vertices twice; that is,
if $s(e_{i})\neq s(e_{j})$ for every $i\neq j$. A cycle $\mu=e_{1}e_{2}%
\cdots e_{n}$ is said to have an \textit{exit} at the vertex
$v=s(e_{1})$, if there is an edge $f\neq e_{1}$ such that $s(f)=v=s(e_{1})$.
\vskip 1ex

We put a binary relation $\geq$ on the set of vertices as follows. If there is
a path from vertex $u$ to a vertex $v$, we write $u\geq v$. A subset $D$ of
vertices is said to be \textit{downward directed} if for any $u,v\in D$,
there exists a $w\in D$ such that $u\geq w$ and $v\geq w$. A subset $H$ of
$E^{0}$ is called \textit{hereditary} if, whenever $v\in H$ and $w\in E^{0}$
satisfy $v\geq w$, then $w\in H$. A hereditary set $H$ is \textit{saturated}
if $r(s^{-1}(v))\subseteq H$ implies $v\in H$, for any regular 
vertex $v\in E^0$.
\vskip 1ex

Intuitively, if there is an edge $e$ with source $v$ and range $w$ then we can
think of $w$ as being an immediate descendant of $v$. The relation $u\ge u'$
just means that $u'$ is a descendant of $u$ and the graph being downward
directed means that every pair of vertices share a common descendant. A subset
being hereditary then means that if an element is in the set then so are all
its descendants and being saturated is the same as saying that 
if all of the immediate descendants of a regular vertex $v$ are in the set
then $v$ is necessarily in the set. \vskip 1ex

\begin{definition}\label{definition-lpa} 
Given an arbitrary graph $E$ and a field $K$, the \textit{Leavitt path algebra}, $L_{K}(E)$,
is defined as follows.  For each $e\in E^{1}$, we create a corresponding \emph{ghost edge}, which we denote $e^{\ast}$. We then define
$r(e^{\ast}):= s(e)$, and $s(e^{\ast}):=r(e)$.  With these data, we then define the Leavitt path algebra on $E$ to be the $K$-algebra generated by a set $\{v:v\in
E^{0}\}$ of pairwise orthogonal idempotents together with a set of variables
$\{e,e^{\ast}:e\in E^{1}\}$ which satisfy the following conditions:
\begin{enumerate}
\item[(1)] $s(e)e=e=er(e)$ for all $e\in E^{1}$;
\item[(2)] $r(e)e^{\ast}=e^{\ast}=e^{\ast}s(e)$\ for all $e\in E^{1}$;
\item[(3)] (The ``CK-1 relations'') For all $e,f\in E^{1}$, $e^{\ast}e=r(e)$ and
$e^{\ast}f=0$ if $e\neq f$;
\item[(4)] (The ``CK-2 relations'') For every regular 
vertex $v\in E^{0}$, we have
$v=\sum_{e\in E^{1},s(e)=v}ee^{\ast}.$
\end{enumerate}
\end{definition} 

We note that CK stands for Cunz-Krieger, who studied these relations in 
the context of graph $C^*$-algebras.  \vskip 1ex

Given a path $\mu=e_{1}e_{2} \cdots e_{n}$ in the graph $E$, we refer to
$\mu^{\ast}=e_{n}^{\ast}\cdots e_{2}^{\ast }e_{1}^{\ast}$ as the corresponding
\textit{ghost path}.\vskip 1ex 

A Leavitt path algebra carries a natural ${\mathbb Z}$-graded structure, where
the vertices have degree zero, the edges have degree one and the ghost edges
have degree $-1$, see  \cite[Lemma 1.7]{AA}. 
A fact that will prove useful to us is that any nonzero graded ideal must 
contain a vertex, see \cite[Corollary 2.4]{ABCR}. 
\vskip 1ex

A subgraph $F$ of a graph $E$ is called \textit{complete} in case
$s_{F}^{-1}(v)=s_{E}^{-1}(v)$, for each regular vertex $v$ in $F^{0}$. (In
other words, if $v$ emits a nonzero finite number of edges in $E$, then all
these edges must belong to $F$). If $F$ is a complete subgraph of $E$ then
$L_{K}(F)$ is a subalgebra of $L_{K}(E)$. This is proved, for example, in
\cite[Corollary 3.3]{Goodearl}. We include a short outline proof here.

\begin{lemma}\label{lemma-subgraph} 
Let $F$ be a complete subgraph of a graph $E$. Then the inclusion of the graph
$F$ in the graph $E$ induces a natural inclusion of $L_K(F)$ as a subalgebra
of $L_K(E)$.
\end{lemma} 

\begin{proof} As the relations for  $L_K(F)$ are a subset of the relations 
for $L_K(E)$, there is a natural induced ring homormorphism from $L_K(F)$ to
$L_K(E)$. The kernel of this map is a graded ideal of $L_K(F)$, as the
homomorphism respects the ${\mathbb Z}$-grading. If the kernel were nonzero
then it would contain a vertex $v$ by \cite[Corollary 2.4]{ABCR}. However, in
this case the vertex $v$ would be equal to zero in $L_K(E)$, a contradiction. 
\end{proof} 

For any vertex $v$, define $T(v)=\{w\in E^{0}:v\geq w\}$.
We say there is a \textit{bifurcation} at a vertex $v$, if $v$ emits more than
one edge. In a graph $E$, a vertex $v$ is called a \textit{line point} if
there is no bifurcation or a cycle based at any vertex in $T(v)$. Thus, if $v$
is a line point, there will be a single finite or infinite line segment $\mu$
starting at $v$ ($\mu$ could just be $v$) and any other path $\alpha$ with
$s(\alpha)=v$ will just be an initial sub-segment of $\mu$. It was shown in
\cite{AMMS} that $v$ is a line point in $E$ if and only if $L_{K}(E)v$ (and
likewise $vL_{K}(E)$) is a simple left (right) ideal and that the ideal
generated by all the line points in $E$ is the socle, ${\rm Soc}(L_{K}(E))$,
of $L_{K}(E)$. \vskip 1ex

We shall be using the following concepts and results introduced by Tomforde
\cite{T}. A \emph{breaking vertex} of a hereditary saturated subset $H$ is an
infinite emitter $w\in E^{0}\backslash H$ with the property that
$0<\#(s^{-1}(w)\cap r^{-1}(E^{0}\backslash H))<\infty$. The set of all breaking
vertices of $H$ is denoted by $B_{H}$. For any $v\in B_{H}$, we define
\begin{equation}
v^H :=v\;-\sum_{s(e)=v,r(e)\not\in H}ee^{\ast}.
\end{equation}
 Given a hereditary saturated subset $H$ and a subset $S\subseteq B_{H}$, we
 say that $(H,S)$ is an \emph{admissible pair}. To an admissible pair $(H,S)$,
 we can associate the ideal generated by $H\cup \{v^{H}:v\in S\}$. We let
 $I(H,S)$ denote this ideal. Tomforde \cite{T} showed that the graded
 ideals of $L_{K}(E)$ are precisely the ideals of the form $I(H,S)$ for some
 admissible pair $(H,S)$. Moreover, $$L_{K}(E)/I(H,S)\cong L_{K}%
 (E\backslash(H,S)).$$ Here $E\backslash(H,S)$ is the quotient graph of $E$,
 whose vertex set is given by $$(E^{0}\backslash H)\cup \{v^{\prime}:v\in
 B_{H}\backslash S\}$$ and whose edges are given by $$\{e\in E^{1} \colon
 r(e)\notin H\}\cup\{e^{\prime}:e\in E^{1},r(e)\in B_{H}\backslash S\},$$ and
 we extend $r,s$ to $(E\backslash(H,S))^{0}$ by setting $s(e^{\prime })=s(e)$
 and $r(e^{\prime})=r(e)^{\prime}$.

\section{Leavitt path algebras satisfying polynomial identities}

In this section we characterize the Leavitt path algebras 
satisfying a polynomial identity and give explicit isomorphisms in the case that we are working with a finite graph.
\vskip 1ex

Let $K$ be a field and let $E$ be an arbitrary graph.  
We show that $L_K(E)$ satisfies a polynomial identity (is PI) 
if and only if no cycle in $E$ has an exit, every path from
a vertex in $E$ ultimately ends at a sink or at a vertex on a cycle and there
is a positive integer $d$ with the property that whenever $\mu$ is a path that does not visit any vertex more than once, then $\mu$ 
necessarily has length at most $d$.
We show that these conditions are in fact equivalent to $L_K(E)$ being a
subdirect product of matrix rings of order $\leq d$ over $K$ and
$K[x,x^{-1}]$. When $E$ is a row-finite graph, $L_K(E)$ is PI if and only
if it decomposes as a direct sum of (possibly infinitely many) matrix rings
over $K$ and $K[x,x^{-1}]$ and each matrix ring in this decomposition is of
dimension at most $d$.\vskip 1ex

We are now ready to describe the Leavitt path algebras satisfying a polynomial
identity.

\begin{theorem} 
\label{General PI LPA} 
Let $K$ be a field and let $E$ be an arbitrary graph.
Then the
following are equivalent:
\begin{enumerate}
\item[(i)] the Leavitt path algebra $L_K(E)$ is a PI algebra; 

\item[(ii)] no cycle in $E$ has an exit, every path from a vertex in $E$
eventually ends at a sink or at a vertex on a cycle, and there is a fixed
integer $d$ such that the number of paths that end at any given sink or on a
cycle (but not including the cycle) is less than or equal to $d$.

\item[(iii)] there exists a fixed positive integer $d$ such that $L_K(E)$ is a
subdirect product of matrix rings over $K$ and $K[x,x^{-1}]$ having order at
most $d$.
\end{enumerate}
\end{theorem}

\begin{proof}(i)$\Rightarrow$(ii). 
Suppose that $L:=L_K(E)$ satisfies a
polynomial identity of degree $N$ and suppose towards a contradiction that
there is a cycle $c$ having an exit $f$ at a vertex $v$. Then $vLv$ is a
subring of $L$ with identity $v$, which satisfies the same polynomial
identity. We then have $c^{\ast}c=v$ in $vLv$. Since $vLv$ satisfies a
polynomial identity we then must have $cc^{\ast }=v$ \cite[Chapt II,
Proposition 4.3]{Procesi}. But we now observe that $c^{\ast}f=0$ and
consequently $cc^{\ast}f=0$. But $cc^{\ast}f = vf=f$, and so we get $f=0$, a
contradiction. Hence no cycle in $E$ has an exit. \vskip 1ex

If there is no integer $d$ satisfying the desired property in (ii), then for
each positive integer $m$ there is some vertex $v$, depending on $m$, such
that there are at least $m^{2}$ paths ending at $v$ (with no repeated
vertices). Now there are two possibilities: either there is some path of
length $>m$ ending at $v$ or for some $i\leq m$ there are $m$ distinct paths
of length $i$ ending at $v$.  We pick $m>N$.\vskip 1ex

As the Jacobson radical of $L_K(E)$ is zero, see \cite[Proposition 6.3]{AA2},
there is a primitive ideal $Q$ such that $v\not\in Q$. Since the prime ring
$L_K(E)/Q$ satisfies the same polynomial identity as $L_K(E)$, we appeal to
Amitsur's theorem \cite[Chapt. II, Theorem 3.1]{Procesi} to conclude that
$L_K(E)/Q$ embeds in a matrix ring $M_{s}(F)$ for some $s\ge 1$ and some field
$F$, and where $s$ is bounded above by $N/2$.\vskip 1ex

Suppose first that we have $m$ distinct paths $\alpha_{1}, \ldots,\alpha_{m}$,
all having the same length $i\leq m$ and ending at $v$. 
Now, the elements $e_{j}:=\alpha_{j}\alpha_{j}^{\ast}$, for $j=1,\dots,m$, are
$m$ mutually orthogonal idempotents in $L$. We claim that their images are 
nonzero and 
distinct in $L/Q$. If not, suppose $e_j\in Q$ or $e_j-e_i\in Q$, for some $i\neq j$. 
Then $e_j=e_j(e_j-e_i)\in Q$, in either case. Hence, 
\[
v=v^2=(\alpha_{j}^{\ast}\alpha_{j})^2= 
\alpha_{j}^{\ast}(\alpha_{j}\alpha_{j}^{\ast})\alpha_{j}
=\alpha_{j}^{\ast}e_j\alpha_{j} \in Q,
\]
a contradiction that establishes our claim. 
However,  $s\le N/2$ and $M_s(F)$ cannot have more than
$s$ orthogonal idempotents. As $m>N>s$, we get a contradiction in this
case.\vskip 1ex

On the other hand, if we have a path $\alpha$ of length $>m$ ending at $v$,
let $\beta_{1},\ldots,\beta_{m}$ be the terminal paths of $\alpha$ such that
$\beta_{i}$ has length $i$ for each $i=1,\ldots,m$, and
$r(\beta_{i})=r(\alpha)=v$. If $f_{i}:=\beta_{i}\beta_{i}^{\ast}$, then the
images of the $f_{i}$ in $L_K(E)/Q$ form a set of $m$ nonzero orthogonal
idempotents in $L_K(E)/Q$. As before, we get a contradiction.\vskip 1ex

Thus we conclude that there is a fixed positive integer $d$ such that, given
any vertex $v$, the number of paths ending at $v$ and having no repeated
vertices is at most $d$.  Since, in addition, cycles to not have exits, this
means that the only infinite paths in $E$ are paths that are
eventually of the form $ggg\cdots$ for some cycle $g$. In other
words, every path from any vertex in $E$ eventually ends at a sink or at a
vertex on a cycle.  This proves that (i)$\Rightarrow$(ii).\vskip 1ex

(ii)$\Rightarrow$(iii). 
Assume (ii). Given any vertex $v$ in $E$, set $$H(v):=\{u\in E^{0}:u\not\geq
v\}$$ and set $$M(v):=E^{0}\backslash H(v)=\{u\in E^{0}:u\geq v\}.$$ If $v$ is
a sink or a vertex on a cycle, then $H(v)$ is a hereditary saturated set and
since $M(v)$ is downward directed, the ideal $P_{v}=I(H(v),B_{H(v)})$ is a
prime ideal not containing $v$, by \cite[Theorem 3.12]{R}. Let $\mathcal{S}$
denote the set of all such (graded) prime ideals $P_{v}$, where $v$ ranges
over all vertices that are either a sink or that lie on a cycle in $E$.
Observe that if $v$ is a sink $E$ then $v$ is still a sink in the quotient
graph $E\backslash(H(v),B_{H(v)})$ and, moreover,
$(E\backslash(H(v),B_{H(v)})^{0}=M(v)$. Hence, in $E\backslash
(H(v),B_{H(v)})$ there are no cycles, every vertex is connected to $v$ by a
path and the number of paths ending at $v$ is at most $d$. This implies that
$E\backslash(H(v),B_{H(v)})$ is a finite acyclic graph with a unique sink $v$.
Hence $L_K(E)/P_{v}\cong L_{K}(E\backslash(H(v),B_{H(v)})\cong M_{n_{v}}(K)$
with $n_{v}\leq d$ (see \cite[Lemma 3.4]{AAS1}). 
Similarly, if $v$ is a vertex on a cycle,
then, by similar argument, $E\backslash(H(v),B_{H(v)})$ is a finite graph in
which every vertex connects to $v$ and $v$ lies on a unique cycle and so, for
the corresponding $P_{v}$, $L_K(E)/P_{v}\cong
L_{K}(E\backslash(H(v),B_{H(v)})\cong M_{m_{v}}(K[x,x^{-1}])$ with $m_{v}\leq
d$ (see \cite[Proposition 3.6]{AAS1}). Now 
$$ 
{\textstyle\bigcap\limits_{P_{v}\in 
\mathcal{S}}} P_{v}=(0).
$$ 
This is immediate due to the fact that every vertex $u$ in $E$
belongs to $M(v)$ for some $v$ where $v$ is a sink or lies on a cycle, and so
for this $v$ we have $u\notin P_{v}$. It follows that the intersection of the
$P_v$ cannot contain any vertices. 
Now, any nonzero graded ideal must contain a vertex, 
by \cite[Proposition 2.4]{ABCR}; 
so the intersection of the (graded ideals) $P_v$ is $0$. 
Thus $L_K(E)$ is a subdirect product of
$\{L/P_{v}:P_{v}\in \mathcal{S}\}$, proving (iii).\vskip 1ex

(iii)$\Rightarrow$(i). This is immediate from the Amitsur-Levitzki theorem
\cite{Procesi}. 
\end{proof}

In the case that the graph $E$ is row-finite, Theorem~\ref{General PI
LPA}(iii)  can be sharpened further, leading to a  structure
theorem for PI Leavitt path algebras over row-finite graphs.

\begin{theorem}
\label{Row finite PI LPA}
Let $E$ be a row-finite graph. Then the following are
equivalent for the Leavitt path algebra $L:=L_{K}(E)$:
\begin{enumerate}
\item[(i)] $L_K(E)$ is a PI algebra;

\item[(ii)] no cycle in $E$ has an exit, every path from a vertex in $E$
eventually ends at a sink or at a vertex on a cycle, and there is a fixed
integer $d$ such that the number of paths that end at any given sink or on a
cycle (but not including the cycle) is less than or equal to $d$.

\item[(iii)] there is a fixed integer $d$ and an isomorphism $$L_K(E)\cong
{\textstyle\bigoplus\limits_{v\in\Lambda}} M_{n_{v}}(K)\oplus
{\textstyle\bigoplus\limits_{C\in\Lambda^{\prime}}} M_{m_{C}}(K[x,x^{-1}]),$$
where $\Lambda$ is the collection of sinks in $E$ and $\Lambda'$ is the
collection of cycles in $E$, and for each $v\in \Lambda$ and $C\in \Lambda'$ 
we have that $n_v$ and $m_C$ are at most $d$;
\item[(iv)] there exists a fixed positive integer $d$ such that, for each minimal
prime ideal $P$ of $L_K(E)$, $L/P$ is isomorphic to a matrix ring over $K$ or
$K[x,x^{-1}]$ of size at most $d$.
\end{enumerate}
\end{theorem}

\begin{proof}
The implication (i)$\Rightarrow$(ii) follows from 
Theorem~\ref{General PI LPA}. \vskip 1ex

(ii)$\Rightarrow$(iii). Assume (ii). One can obtain (iii) from (ii) 
by following the
proof of Theorem 3.9 of \cite{AAPS}. Here we outline a 
proof. If $(w_{i}:i\in I\}$ is the set of all the sinks in $E$, 
then the socle
of $L$ is $S= {\textstyle\bigoplus\limits_{i\in I}} S_{i}$ where each ideal
$S_{i}= {\textstyle\bigoplus\limits_{r=1}^{n_i}}
Lw_{i}\alpha_{i_{r}}^{\ast}$ with $n_{i}$ being the number
of paths $\alpha_{i_{r}}$ in $E$ that end at the sink $w_{i}$, see 
\cite{AMMS}.
Now $S_{i}\cong M_{n_{i}}(K)$, as $S_{i}$ is a direct sum of $n_{i}$
isomorphic simple modules (whose endomorphism ring is the field $K$).
Likewise, if $\{c_{j}:j\in J\}$ is the set of all distinct cycles (with no
exits) in $E$ and if $T$ the ideal generated by all the vertices in these
cycles, then $T= {\textstyle\bigoplus\limits_{j\in J}} T_{j}$. Here $T_{j}$ is
the ideal generated by the vertices on a single no-exit cycle $c_{j}$ based at
a vertex $v_{j}$ and so $T_{j}=M_{l_{j} }(K[x,x^{-1}])$ where $l_{j}$
represents the number of paths that end at $v_{j}$ but do not include that
cycle $c_{j}$ \cite[Propositions 3.6, 3.7]{AAS1}. 
Consider the ideal $S+T$. Now $H=(S+T)\cap E^{0}$ is
a hereditary saturated set. Since every path from a vertex $u$ in $E$
eventually ends in a sink or at a vertex on a cycle, the row-finiteness of $E$
implies (by a simple induction on the length of a path of maximum length from
$u$ to a vertex in $H$) that $u$ belongs to the hereditary saturated set $H$.
Thus $E^{0}\subset H$ and we conclude that the ideal $S+T=L$. Also $S\cap T=0$
since if a vertex belongs to the hereditary saturated set $S\cap T$, this will
give rise to a cycle $c$ which will have an exit, a contradiction. Thus
$L=S\oplus T= {\textstyle\bigoplus\limits_{i\in\Lambda}} M_{n_{i}}(K)\oplus
{\textstyle\bigoplus\limits_{j\in\Lambda^{\prime}}} M_{l_{j}}(K[x,x^{-1}])$
with the desired properties for $n_{i},l_{j} ,\Lambda,\Lambda^{\prime}$.\vskip
1ex

(iii)$\Rightarrow$(iv). Assume (iii). Now any minimal prime ideal $P$ of $L$
is the complement of a single matrix summand in $L$, namely, $L=P\oplus Q$
where $Q\cong M_{n_{i} }(K)$ or $M_{l_{j}}(K[x,x^{-1}])$ where $n_{i},l_{j}\leq
d$. This proves (iv)$.$\vskip 1ex

(iv)$\Rightarrow$(i). 
Assume (iv). Since the Jacobson radical of $L$ is $0$, the intersection of all
minimal prime ideals $P$ of $L$ is $0$. Thus $L$ embeds in the direct product
of $L/P$ for various minimal prime ideals $P$ of $L$. By hypothesis, each
$L/P$ is a matrix ring of size $\leq d$ over $K$ or $K[x,x^{-1}]$ and so
satisfies the standard polynomial identity $S_{2d}$ by the Amitsur-Levitzki
theorem \cite{Procesi}. Then $L$ itself satisfies the polynomial identity
$S_{2d}$, thus proving (i).
\end{proof}

We give an example of a graph  $E$ that is not row-finite whose 
Leavitt path algebra $L_{K}(E)$ satisfies a polynomial identity 
but is not a direct sum
of matrix rings over $K$ and $K[x,x^{-1}]$.

\begin{example}
Consider the graph $E$ whose vertex set consists of a vertex $w$ along with a
countably infinite set of vertices $v_{1},v_2,\ldots$, and with a directed edge from $w$ to each $v_i$ 
and with a loop based at each $v_i$.   The graph $E$ is given below.%
\vskip 5mm

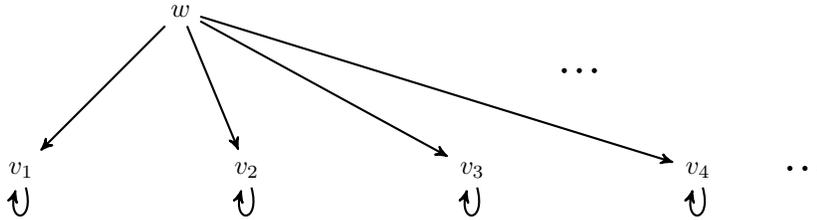
\begin{tikzpicture}[->,>=stealth',shorten >=1pt,auto,node distance=3cm,
 thick,
main node/.style={     circle,
        draw,
        solid,
        fill=black!100,
        inner sep=0pt,
        minimum width=4pt}]


  \node (1) {$w$};
  \node (2) [below left of=1] {$v_1$};
  \node (4) [right of=2] {$v_2$};
  \node (3) [right of =4] {$v_3$};
  \node (5) [right of =3] {$v_4$};
  
  \draw[fill=black] (5.1,-0.8) circle (.02);
  \draw[fill=black] (5.3,-0.8) circle (.02);
  \draw[fill=black] (5.5,-0.8) circle (.02);

  \draw[fill=black] (8.1,-2.1) circle (.02);
  \draw[fill=black] (8.3,-2.1) circle (.02);
  \draw[fill=black] (8.5,-2.1) circle (.02);

 \path[every node/.style={font=\sffamily\small}]

   (1) edge node [right] {} (2)
    (1) edge node [right] {} (3)
     (1) edge node [right] {} (4)
      (1) edge node [right] {} (5)
 
      (2)  edge [loop below] node {} (2)
       (3)  edge [loop below] node {} (3)
        (4)  edge [loop below] node {} (4)
         (5)  edge [loop below] node {} (5);

\end{tikzpicture}

\end{example}

Using the notation of Theorem \ref{General PI LPA}, for each $i$,
$M(v_{i})=\{w,v_{i}\}$, $$H(v_{i})=\{v_{j}:j\neq i\}$$ and $P_{v_{i}}$ is the
(prime) ideal generated by $$H(v_{i})\cup\left\{w^{H(v_{i})}\right\}$$ and
$P_{w}$ is the (maximal) ideal generated by $\{v_{1},v_2,\ldots \}$. 
For each
$i$, $L/P_{v_{i}}\cong M_{2}(K[x,x^{-1}])$ and $L/P_{w}\cong K$. 
Hence, $L$ is the
subdirect product of $K$ and infinitely many copies of $M_{2}(K[x,x^{-1}])$.
Thus $L$ is a PI algebra, but $L$ cannot decompose as a direct sum of matrix
rings over $K$ and $K[x,x^{-1}]$, since if this were the case, $w$ would lie
in a direct sum of finitely many matrix rings of finite order; since the ideal
generated by $w$ is all of $L_K(E)$, we would then necessarily have that
$L_K(E)$ embeds in a finite direct sum of matrix rings of finite order over
$K$ and $K[x,x^{-1}]$, contradicting the fact that $L_K(E)$ has infinitely
many orthogonal idempotents.

In the next section, we explore the connection between polynomial identities
and GK dimension for Leavitt path algebras.



\section{Leavitt path algebras with GK dimension $\leq1$}

We first show that if $E$ is a finite graph, then $L_{K}(E)$ is a PI algebra
if and only if the Gelfand-Kirillov dimension (for short, GK dimension) of
$L_K(E)$ is $\leq1.$  (Note that, in view of Theorem~\ref{Row finite PI LPA}, 
this can be deduced
from work in \cite{AAJZ1, AAJZ2}.) 
Examples are constructed showing that this statement is
no longer true if $E$ is an infinite graph. \vskip 1ex

Let $K$ be a field, let $A$ be a finitely generated $K$-algebra, and let $V$
be a finite-dimensional subspace of $A$ that generates $A$ as a $K$-algebra.
Then the \textit{Gelfand-Kirillov dimension} of $A$ (GK dimension, for short)
is defined by
\[
\text{GKdim}(A)=\underset{n\rightarrow\infty}{\lim\sup}\log_{n}
(V+V^2+\cdots +V^n), 
\]
where $V^i$ denotes the subspace of $A$ spanned by all products of $i$
elements from $V$. We note that GKdim($A$) is independent of the choice of
generating subspace $V$.\vskip 1ex

If $A$ is not finitely generated as a $K$-algebra, then we define the GK
dimension of $A$ by \[ \text{GKdim}(A)={\rm sup}_B ~\text{GKdim}(B), \] where
$B$ runs over all finitely generated $K$-subalgebras of $A$. For basic
properties and results about GK dimension, we refer the reader to \cite{KL}.
For Leavitt path algebras of finite graphs, we show that having
Gelfand-Kirillov dimension at most one is equivalent to satisfying a
polynomial identity. 

\begin{theorem}
\label{GK-Finite graph} Let $E$ be a finite graph and $L=L_{K}(E)$. Then the
following are equivalent:
\begin{enumerate}
\item[(i)]  $L_K(E)$ is a PI algebra;

\item[(ii)] no cycle in $E$ has an exit;

\item[(iii)] $L$ is a direct sum of finitely many matrix rings of finite order over
$K$ and $K[x,x^{-1}]$;

\item[(iv)] $L$ has GK dimension $\leq1$;

\item[(v)] $L$ is a finite module over its centre;

\item[(vi)] if $a,b\in L$ satisfy $ab=1$ then $ba=1$
\end{enumerate}
\end{theorem}

\begin{proof}
The equivalence of (i)--(iii) follows from Theorem
\ref{Row finite PI LPA} specialized to the case of a finite graph. Since ${\rm GKdim}(K)=0$ and
${\rm GKdim}(K[x,x^{-1}])=1$, we immediately have (iii)$\implies$(iv).  Similarly, (iii) implies (v). 
The fact that (iv) implies (i) follows immediately from the Small-Stafford-Warfield theorem \cite{SSW} and (v) implies (i) is immediate from 
basic facts about PI algebras, so that (i)--(v) are equivalent. \vskip 1ex

Another general fact about PI algebras, \cite[Chapter II, Proposition 4.3]{Procesi}, shows that  (i)$\Rightarrow$(vi). 
Finally, we show that  (vi)$\Rightarrow$(ii). Suppose that (vi) holds but that there is a cycle $c$ based at $v$ with an edge $e$ 
that starts at $v$ and does not lie on $c$. Note that $c^*c=v$. Set $u$ to be the sum of all vertices in the graph other than $v$, and note that $u+v=1$. 
Using the fact that the distinct vertices give orthogonal idempotents in $L$, we see that $(u+c^*)(u+c)=u+v=1$, and so $(u+c)(u+c^*)=1$, by (vi). 
Thus, $u+v=1=(u+c)(u+c^*)= u+cc^*$ and so $v=cc^*$. Hence, $e=ve=c(c^*e)=0$, a contradiction. 
\end{proof}

In general the conclusion to Theorem \ref{GK-Finite graph} need not hold for
arbitrary infinite graphs. Specifically, there exist infinite graphs $E$ for
which the GK dimension of $L_K(E)$ is either zero or one, but where $L_K(E)$
is not a PI algebra.

\begin{example} 
\label{example-non-pi}
Let $E$ be the infinite row-finite graph given below.\\[1ex]

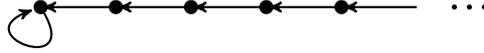
\begin{center}
\begin{tikzpicture}[->,>=stealth',shorten >=1pt,auto,node distance=3cm,
 thick,main node/.style={     circle,
        draw,
        solid,
        fill=black!100,
        inner sep=0pt,
        minimum width=4pt}]
\tikzset{every loop/.style={min distance=10mm,in=200,out=300,looseness=11}}

  \node[main node] (1) {}
  (1) edge [loop above] node {} (1);

  \draw[fill=black] (5.9,0) circle (.02);
   \draw[fill=black] (5.7,0) circle (.02);
    \draw[fill=black] (5.5,0) circle (.02);
   
         \draw[fill=black] (0,0) circle (0.08);
           \draw[fill=black] (1,0) circle (0.08);
              \draw[fill=black] (2,0) circle (0.08);      
                \draw[fill=black] (3,0) circle (0.08);
                  \draw[fill=black] (4,0) circle (0.08);
                  
  \draw (1,0)--+(-1,0); 
   \draw (2,0)--+(-1,0); 
    \draw (3,0)--+(-1,0);       
     \draw (4,0)--+(-1,0);    
      \draw (5,0)--+(-1,0);
             

 \end{tikzpicture}
\end{center} 

\vspace{2ex} 

As $E$ is row finite, but fails to satisfy
Condition (ii) of Theorem~\ref{Row finite PI LPA}, since there is no suitable
integer $d$, we see that $L$ is not a PI algebra. \vskip 1ex

For each $n\geq 2$, let $E_n$ be the subgraph of $E$ that contains the first
$n$ vertices from the left, and all edges to the left of the $n$th vertex.
Then $E_n$ is a finite graph with one cycle and this cycle has no exit; hence,
${\rm GKdim}(E_n)= 1$, by \cite[Theorem 5]{AAJZ1}. As each $E_n$ is a complete
subgraph of $E$, each $L_{K}(E_{n})$ can be naturally viewed as a subalgebra
of $L_K(E)$. Moreover, $L_K(E)$ is the directed union of the $L_{K}(E_{n})$
and so ${\rm GKdim}(E)= 1$.

\end{example}

\begin{example}
Let $F$ be the graph obtained from $E$ in the above example by removing the
vertex with the loop and the two edges that end at this vertex. Then $F$
contains no cycles and arguments similar to those used in the preceding
example shows that the corresponding Leavitt path algebra $L_{K}(F)$ has
GK dimension $0$, but is not a PI algebra.
\end{example}

In view of the preceding examples, one would like to investigate the nature of
a Leavitt path algebra having GK dimension $0$ or $1$.\vskip 1ex

We first consider the Leavitt path algebras whose GK dimension is $0$.

\begin{theorem}
\label{GK-dim 0}
Let $E$ be an arbitrary graph. Then the following are
equivalent for the Leavitt path algebra $L:=L_K(E)$:
\begin{enumerate}

\item[(i)] $L$ has GK dimension zero; 

\item[(ii)] the graph $E$ has no cycles;

\item[(iii)] $L$ is von Neumann regular and is a directed union of subalgebras
each of which is a direct sum of finitely many matrix rings of finite order
over $K$. 
\end{enumerate} 
\end{theorem} 

\begin{proof} 
(i)$\Rightarrow$(ii). Suppose that $E$ has a cycle $c$ based at a vertex $v$.
Let $V:= Kv\oplus Kc$. As the powers of $c$ are linearly independent over $K$,
we see that $\dim(V^n)\geq n$. This forces $\text{GKdim}(L)\geq 1$, and so
(i)$\Rightarrow$(ii) follows. \vskip 1ex

(ii)$\Rightarrow$(iii). This is proved in \cite[Theorem 1]{AR}.\vskip 1ex

(iii)$\Rightarrow$(i). This follows from the defintion of the GK dimension 
of an arbitrary $K$-lagebra in terms of finitely generated subalgebras, and the 
fact that matrix rings over $K$ have GK dimension zero. 
\end{proof}

Next, we consider the Leavitt path algebras with GK dimension one.

\begin{theorem}
\label{GK-dim 1}
Let $E$ be an arbitrary graph. Then the following are
equivalent:
\begin{enumerate}
\item[(i)] $L_K(E)$ has GK dimension at most one;

\item[(ii)] no cycle in $E$ has an exit;

\item[(iii)] $L_K(E)$ is a directed union of subalgebras each
of which is a direct sum of finitely many matrix rings of finite order over
$K$ and $K[x,x^{-1}]$;
\item[(iv)] $L_K(E)$ is locally PI.
\end{enumerate}
Moreover, if $L_K(E)$ has GK dimension $\le 1$, then it has 
GK dimension zero if and only if $E$ has no cycles; otherwise, 
it has GK dimension one.
\end{theorem}

\begin{proof} 

(i)$\Rightarrow$(ii). 
Suppose that $E$ has a cycle $c$ based at a vertex $v$ and that there is an
edge $f$ with source $v$ that does not lie on $c$. We will show that the set
$S:=\{c^i (c^{\ast})^j \colon i,j\ge 0\}$ is a linearly independent set. The
implication (i)$\Rightarrow$(ii) follows immediately, as the number of pairs
$(i,j)$ with $i+j\leq n$ is quadratic in $n$; so that $\text{GKdim}(L_K(E)\geq
2$. \vskip 1ex 

We will use the following, easily checked, facts: $c^*c=v,\, vf=f,\, cv=c,\,
c^*f=0$, and note that, by convention, $c^0=(c^*)^0=v$.  \vskip 1ex 

Suppose that $S$ is not a linearly independent set and consider a nontrivial
relation \[ \sum_{j=0}^n\,a_j(c^*)^j =0 \] where each $a_j$ is in the
$K$-subalgebra generated by $c$. Let $t$ be the least integer such that
$a_t\neq 0$. Multiply the above equation on the right by $c^tf$ to obtain 
\[
a_tf + \sum_{j=t+1}^n\,a_j(c^*)^{j-t}f =0
\]
Now, each $(c^*)^{j-t}f=0$, as $j-t>0$ and $c^*f=0$. Hence, $a_tf=0$. 
Write $a_t=\sum_{i=1}^m\,k_ic^i$ with $k_m\neq 0$; so that 
\[
\sum_{i=1}^m\, k_ic^if=0.
\]

Now, $(c^*)^mc^if=(c^*)^{m-i}f=0$, for each $i<m$; so multiplying the above
equation on the left by $(c^*)^m$ gives $k_mf=0$, contradicting the fact that
$k_m\neq 0$. Thus, $S$ is a linearly independent set, and (i)$\Rightarrow$(ii)
is established.  \vskip 1ex 

(ii)$\Rightarrow$(iii). In Proposition 2 of \cite{AR}, it was shown that a Leavitt path algebra
$L$ over an arbitrary graph $E$ is a directed union of subalgebras $B$, where
each $B=im(\theta)\oplus$ (a finite direct sum of copies of $K$) and where
$\theta$ is a graded monomorphism $L_{K}(E_{F})\longrightarrow L$. Here
$E_{F}$ is a finite graph constructed from a prescribed finite set $F$ of
edges in $E$. Moreover, if no cycle in $E$ has an exit, then it is clear from
its construction that the finite graph $E_{F}$ also has the same property.
Thus given any finite set $F$ of edges in $E$, no cycle in the corresponding
finite graph $E_{F}$ has an exit in $E_{F}$. We then appeal to Theorem
\ref{GK-Finite graph} to conclude that $im(\theta)$ $\cong L_{K}(E_{F})$ 
is a direct sum of finitely many matrix rings over
$K$ and $K[x,x^{-1}]$, and hence so is the subalgebra $B$. 
\vskip 1ex

(iii)$\Rightarrow$(iv). Assume that $L_K(E)$ is a directed union of
subalgebras each of which is a direct sum of finitely many matrix rings of
finite order over $K$ and $K[x,x^{-1}]$. Then any finitely generated
subalgebra of $L_K(E)$ will be contained in a such a subalgebra and so will be
PI.  \vskip 1ex

(iv)$\Rightarrow$(i). 
This follows from the definition of the GK dimension of
an arbitrary $K$-algebra and the fact that direct sums of matrix rings over a
field $K$ have GK dimension $0$ and direct sums of matrix rings over
$K[x,x^{-1}]$ have GK dimension one. \vskip 1ex

Finally, Theorem~\ref{GK-dim 0} shows that if 
$\text{GKdim}(L_K(E))\leq 1$ then $\text{GKdim}(L_K(E))=0$ if and only 
if $E$ has no cycles. 
\end{proof}

\end{document}